\documentclass{amsart}
\usepackage{amsmath, amssymb, amsthm}
\usepackage{color}

\numberwithin{equation}{section}

\theoremstyle{plain}
\newtheorem{prop}{Proposition}[section]

\newtheorem{coro}[prop]{Corollary}
\newtheorem{lemm}[prop]{Lemma}

\newtheorem{theo}[prop]{Theorem}

\theoremstyle{definition}
\newtheorem{defi}[prop]{Definition}

\newtheorem{exam}[prop]{Example}

\newcommand\goesR[1]{\overset{{}_{#1}}{\rightsquigarrow}_{\hspace{-0.4ex}{\scriptscriptstyle R}}}

\newcommand\gozoi{\mathrel{\overset{\scriptscriptstyle 0,1,\infty}{\rightsquigarrow}}}
\newcommand\gozoiR{\mathrel{\gozoi_{\hspace{-1.2ex}{\scriptscriptstyle R}}}}

\newcommand\gozot{\mathrel{\overset{\scriptscriptstyle 0,1,2}{\rightsquigarrow}}}
\newcommand\gozotR{\mathrel{\gozot_{\hspace{-0.9ex}{\scriptscriptstyle R}}}}

\newcommand\resp{{\it resp}}

\begin{document}

\author{Eddy GODELLE and Sarah REES}

\address{Eddy Godelle, Normandie Univ, France; UNICAEN, LMNO, F-14032 Caen, France; CNRS UMR 6139, F-14032 Caen, France}
\email{eddy.godelle@unicaen.fr}
\urladdr{//www.math.unicaen.fr/\!\hbox{$\sim$}godelle}

\address{Sarah Rees, School of Mathematics and Statistics, University of Newcastle, Newcastle NE1 7RU, UK}
\email{sarah.rees@ncl.ac.uk}
\urladdr{//www.mas.ncl.ac.uk/\!\hbox{$\sim$}nser}

\title{Rewriting systems in sufficiently large Artin--Tits groups}



\begin{abstract}
 A conjecture of Dehornoy  claims that, given a presentation of an Artin-Tits group,  every word that represents the identity can be  transformed into the trivial word
using the braid relations, together with certain rules
(between pairs of  words that are not both positive)
that can be derived directly from the braid relations, as well as free
reduction, but
 without introducing trivial factors $ss^{-1} $ or $s^{-1} s$. This conjecture
is known to be true for Artin-Tits groups of spherical type or of FC type. We prove the conjecture for Artin--Tits groups of sufficiently large type.
\end{abstract}

\thanks{Work partially supported by GDRI 571: French-British-German network in Representation theory}

\maketitle

\section{Introduction}
Let $\Gamma_0$ be a  labelled finite complete simplicial graph with vertex set $S$, edge set~$E$ and label map $m: E\to \mathbb{N}_{\geq 2}\cup\{\infty\}$.   
The associated \emph{Coxeter graph}~$\Gamma$ is then obtained from~$\Gamma_0$ by removing edges labelled by $2$, and unlabelling edges whose label is $3$. 
One associates with $\Gamma_0$ an \emph{Artin--Tits group} (also known as \emph{Artin group}), denoted by~$A(\Gamma_0)$ (or by~$A(\Gamma)$), defined by the following \emph{Artin group presentation}:
$$\left\langle\  S\ \left | \underbrace{sts\cdots}_{m} = \underbrace{tst\cdots}_{m}\ ;\ s,t\in S \textrm{ and } m(\{s,t\}) = m\neq \infty \right.\right\rangle.$$

For instance, if $\Gamma_0$ is a complete graph with all edges labelled with $\infty$, then $\Gamma$ is equal to $\Gamma_0$ and  the group~$A(\Gamma)$ is the free group on $S$.
Alternatively, if all the labels in~$\Gamma_0$ are equal to $2$, then $\Gamma$ is completely disconnected and the group~$A(\Gamma)$ is the free abelian group on $S$.
More generally, considering graphs~$\Gamma_0$ whose labels are $2$ or $\infty$ only, one obtains the so-called right-angled Artin groups (RAAG for short), which are the object of numerous articles.
Finally, starting from any (undirected) Dynkin diagram and replacing doubled and tripled edges by edges labelled by $4$ and $6$, respectively, one obtains a Coxeter graph and therefore an Artin group, which is helpful in the study of the associated Coxeter group and the associated Hecke algebra.
For instance, if $\Gamma$ is an unlabelled line and $S$ has cardinality~$n$, then $A(\Gamma)$ is the braid group on $n+1$ strings.

So, the family of Artin groups contains various groups that are of interest. In the present article, we mainly focus on Artin groups where all the labels in $\Gamma_0$ are either equal to $\infty$ or greater or equal to $3$.
These groups are called \emph{large type Artin groups} and have been considered by various authors~\cite{App,ApS,BrC,HoR}.    

Artin groups are badly understood in general and few results are known. For instance, the word problem is an open problem for Artin groups, although it has been solved for particular subfamilies, including all those introduced above. In general, starting from a finite presentation of a group, there is no general
strategy to solve its word problem.
From the finite presentation, we can always derive
a finite number of elementary transformations on words (that replace some factors of words by other words -- see the Preliminary
Section~\ref{sec:prelim} for the definition of a factor), each of which
transforms any word to which it applies into another word with strictly shorter length.
And, whenever a sequence of such transformations is applied to an input word $w$,
because the sequence of word lengths is strictly decreasing,
this process must stop after at  most $|w|$ such elementary transformations.
Where the group presentation is hyperbolic,
the resulting word is the empty word if and only if the initial word represents
the identity element of the group. So, using the bound on the number of possible
successive elementary transformations, one gets a solution to the word problem. 
But, when the group presentation is not hyperbolic, there must be 
representatives of the identity element that are not reduced by the finite
set of transformations to
the empty word, and hence the process does not solve the word problem.

Unfortunately, except in the case of free groups, Artin groups are not
hyperbolic as they contains copies of $\mathbb{Z}^2$, so the approach above
does not work to solve the word problem in Artin groups. However, one might expect that
an alternative family of elementary transformations could lead to a solution to the word problem for any Artin group. It is natural to look for a solution
with a set of elementary transformations on words that do not increase length. 
For instance, the solution to the word problem obtained in~\cite{Ser} for right-angled Artin groups can be seen as a special example of this strategy. 
\begin{defi}
\label{D:Special}
For a finite presentation  $\langle S, R\rangle$, we define \emph{special transformations} on finite words on~$S\cup S^{-1}$ as follows:\\
- \emph{type $0$} (\resp. $\infty$): Remove (\resp. insert) some pair~$s^{-1} s$ or~$s s^{-1} $, where $s$ is a letter of~$S$,\\
- \emph{type~$1$}: Replace some factor $v$ by~$u$, or $v^{-1} $ by~$u^{-1} $, where $u = v$ is a relation of~$R$.
\end{defi} 

Since the presentation is finite, the set of special transformations is finite.
But note that, throughout the paper, we shall abuse notation and consider relations $u=v$ and $v=u$
to be the same relation, and hence, in particular, we shall also consider that 
there are four (rather
than two) type $1$ transformations associated with each relation $u=v$.

 From now on, for a set~$X$ of types of transformations on words, we write $w \goesR{\phantom1X\phantom1} w'$ if $w$ can be transformed to~$w'$ using transformations of type~$X$. Thus, for a group~$G$  with finite presentation~$\langle S, R\rangle$ and for every word~$w$ on~$S\cup S^{-1}$, we have
\begin{equation} w \gozoiR \varepsilon \quad \iff \quad \overline{w}=1,\end{equation} 
where $\varepsilon$ is the empty word and $\overline{w}$ is the image of $w$ in the group~$G$.
Of course here we do not get a solution to the word problem as long as we do not have a bound on the length of the path from $w$ to~$\varepsilon$ using relations of types $0,1,\infty$ --- such a bound does not exist in general. 
If one wants to obtain a set of elementary transformations that do not increase the length then it is legitimate to include transformations of type~$0$ or $1$, whereas transformations of type~$\infty$ have to be avoided.
Clearly if $w\goesR{0,1} \varepsilon$ then $\overline{w} = 1$ but the converse is not true in general, nor in the special case of Artin group presentations (see~\cite{DeG} for an example). 
It is therefore natural to address the question of whether one may add some other special transformations to obtain an equivalence for all Artin groups. Two partial results on that direction appear recently in the literature~\cite{DeG,HoR}. 

In~\cite{HoR} a family of elementary operations, called \emph{of type $\beta$}, is introduced. A special case of the main result  proved in~\cite{HoR}  is that for any Artin group presentation~$\langle S, R\rangle$ of large type and for every word~$w$ on~$S\cup S^{-1}$, we have  
\begin{equation} w \goesR{0,1,\beta} \varepsilon \quad\iff \quad \overline{w}=1.\end{equation}
The type $\beta$ operations do not
increase the length, and so this leads to a solution of the word problem.
However type $\beta$ operations are not finite in number (even if for a
given length only a finite number of them can be applied) and given a word,
it is not easy to detect where a type $\beta$ transformation can be applied.
Moreover, as shown in~\cite{HoR}, one cannot expect that the above equivalence holds for every Artin group.

Now comes the crucial point. 
In~\cite{DeG}, some special transformations of \emph{type $2$} are introduced,
falling into two types called     
\emph{type $2r$} (``type~$2$-right'') and
\emph{type $2l$} (``type~$2$-left  ),
which allow some limited substitution using unsigned words.
These are defined as follows:
\begin{defi}
\label{D:Special2}
Let $G=\langle S,R\rangle$ be a group, and suppose that
$u=v$ is a relation in $R$, between two positive words $u,v$.
Suppose that $u$ can be written as a concatenation $u_1u_2$ and $v$
as a concatenation $v_1v_2$.
Suppose that $w$ is a word over $S \cup S^{-1}$.
We define a transformation
that replaces a factor $v_1^{-1}u_1$ of $w$ by $v_2u_2^{-1}$
to be of \emph{type $2r$} (``type-$2$-right''), provided that
both $v_1$ and $u_1$ are non-empty.
We define a transformation
that replaces a factor  $v_2u_2^{-1}$ of $w$ by $v_1^{-1}u_1$
to be of \emph{type $2l$} (``type-$2$-left''), provided that
both $v_2$ and $u_2$ are non-empty.
\end{defi} 
Basically, the rules of type~$2$ are those directly derived from braid relations for which the
left hand side of the rule is unsigned.
Note that in the special case where all four of $u_1,u_2,v_1,v_2$ are
non-empty, the associated transformations of types $2l$ and $2r$ are
inverse to each other.

\begin{exam} Consider  a \emph{dihedral} Artin group with $S = \{s,t\}$ and $m_{s,t} = 3$. 
Then the transformations
  \[ s^{-1}t^{-1}s \rightarrow ts^{-1}t^{-1}\,\hbox{\rm and}\,
t^{-1}s^{-1}t^{-1}s \rightarrow s^{-1}t^{-1}\]
are both of type $2r$,
the first one with $u_1=s$, $u_2=ts$, $v_1=ts$ and $v_2=t$,
and the second one with $u_1=s$, $u_2=ts$,$v_1=tst$ and $v_2$ empty.
Then, the reverse of the first, \[ts^{-1}t^{-1}\rightarrow s^{-1}t^{-1}s\]
is type $2l$, but the reverse of the second,
\[ s^{-1}t^{-1}\rightarrow t^{-1}s^{-1}t^{-1} s\] is not of type $2l$.
\end{exam}

It is immediate that the set of transformations of type~$2$ is finite, and
therefore so is the set of transformations of types~$0,1$ or $2$.  Moreover, it is clear that $w \goesR{0,1,2}\varepsilon$ implies~$\overline{w}  = 1$.
The reader may wonder why operations of type~$2$ should be considered of special interest. 
The explanation comes from  the case of \emph{spherical type} Artin presentations
(that is, those based on finite type Dynkin diagrams, and so in particular all braid groups).

In this case, $w \goesR{0,1} \varepsilon \iff \overline{w}  = 1$ does not hold,but  $w \goesR{0,2} \varepsilon \iff \overline{w}  = 1$ does~\cite{DeG}.
Further, in such a group, a sequence of operations of types~$2r$ (resp. $2l$)
and $0$ allows one to transform any (unsigned) word $w$ into a word of the
form $w_1w_2^{-1}$ (resp. $w_1^{-1}w_2$) for which $w_1,w_2$ are positive
words with no common right (resp. left) divisor in the associated Artin-Tits monoid.
(See \cite{DeG} and Propositions~ \ref{leftnormaldecomp} and \ref{PropConvNF} 
below.)

The result we are going to prove here is 
\begin{theo}\label{TH1} For every sufficiently large Artin-Tits presentation~$\langle S, R\rangle$, and any word $w\in S\cup S^{-1}$,  
 \begin{equation}w \goesR{0,1,2} \varepsilon \iff \overline{w}  = 1\end{equation}
\end{theo}
Our definition of sufficiently large type for an Artin-Tits presentation is taken from \cite{HoR2}, but we repeat it as Definition~\ref{def:sufflarge}; note that this definition includes all presentations of large type.

It is conjectured in~\cite[Conj.~3.28]{Deh}
that the equivalence of Theorem~\ref{TH1} holds for all Artin-Tits presentations. 
The conjecture is repeated in ~\cite[Conj.~1.6]{DeG}, where it is proved for Artin-Tits presentations of spherical and FC types.

\begin{defi}
A finite presentation~$\langle S, R\rangle$ of a group~$G$ satisfies Property~$H$ if the equivalence $w \gozotR \varepsilon \iff \overline{w} = 1$ holds for every word~$w$ on~$S\cup S^{-1}$.
\end{defi}
So, we can rephrase \cite[Conj.~1.6]{DeG} as a claim that Property~$H$ is satisfied by every Artin presentation, and Theorem~\ref{TH1} as a verification of
that claim for any Artin presentation of sufficiently
large type.
The conjecture was already proved in \cite{DeG} for spherical type Artin groups,
and then more generally for all Artin groups of FC type, and so all right-angled Artin groups.
But most Artin groups of sufficiently large type (and in particular those of
large type) are not of FC type, and so are not covered by the results of
\cite{DeG}; so certainly our result provides a new family of Artin groups
for which the conjecture is valid.
We prove the result for groups of sufficiently large type as an extension of
the result for groups of large type, recognising that the first of those two
classes of groups can be  built out of the second in the same way that the
groups of FC type can be built out of the groups of spherical type;
so this part of our proof follows the proof of \cite{DeG}.
In order to obtain their result for Artin groups of FC type, the authors of \cite{DeG} proved that a stronger but more technical property,
namely Property~$H^\#$ (defined in \cite{DeG}), is satisfied by Artin groups of spherical type. We will prove that this property holds for Artin groups of large type. Then, combining this result with~\cite[Prop.~2.14]{DeG} we will derive the following. 
\begin{theo}
\label{TH2} Let $\mathcal{T}$ be the family of those Artin-Tits presentations
$(S,R)$ for which every connected full subgraph of the Coxeter graph~$\Gamma$
without an $\infty$-labelled edge is either of spherical type or of large type. 
Then
every presentation $(S,R)$ in $\mathcal{T}$ satisfies Property~$H^\#$ and (therefore) Property~$H$. Hence, for every word $w$ on~$S\cup S^{-1}$,   
 $$w \goesR{0,1,2} \varepsilon \iff \overline{w}  = 1.$$
\end{theo}      

The reader should note that transformations of type~2 can increase the length
of words  and therefore our result does not lead directly to a solution to the word problem. However,
whenever we can prove a bound on the maximal number of iterations of special
transformations we can derive a solution to the word problem. We might expect 
that such a bound could be found in all cases, even if a general argument for that
remains to be found. The main argument of our proof will consist of 
a proof that that any transformation of type $\beta$ can be decomposed as a
sequence of type~$2$ and type~$1$ transformations. 
The remainder of this article is organised as follows.
Section~\ref{sec:H} is devoted to the proof of Property~$H$ for Artin-Tits
groups of large type, stated at the end of the section as Corollary~\ref{CoroHLarge2}. 
Section~\ref{sec:sufflarge} considers Artin-Tits groups of sufficiently large type.
Those groups are introduced at the beginning of 
Section~\ref{sec:sufflarge}, and characterized in
Section~\ref{sec:char_sufflarge} as a particular class of Artin groups closed under
amalgamation over standard parabolic subgroups.
Property~$H^\#$ is defined in Section~\ref{sec:Hsharp}, and Proposition~\ref{propHDpAmal} \cite[Proposition 2.14]{DeG} is stated, which will be used to extend the
proof of Property~$H^\#$ for large type groups to a larger class of Artin-Tits
groups, which (according to the given characterization) must include all those of sufficiently large type. 
Section~\ref{secCritSequ} introduces the machinery of critical factorizations
and critical sequences of transformations for large type groups from \cite{HoR}.
Section~\ref{secTransv} is devoted to the proof that Artin groups of large type
satisfy Property~$H^\#$, and the final section, Section~\ref{sec:finalproof}, to the proofs
of Theorems~\ref{TH2} and ~\ref{TH1}.

We would like to acknowledge some helpful comments from the referee, which led us to introduce property $H^+$, strengthen Corollary~\ref{CoroHLarge2} as Corollary~\ref{CoroHplusLarge}, and prove Proposition~\ref{propHplus_amalg}.
 
\section{Property~$H$ For Artin-Tits groups of large type}
\label{sec:H}
In this section we prove that Property~$H$ holds for Artin-Tits group of large type. Throughout this section, we fix a  Coxeter graph~$\Gamma$. We denote
by~$\Gamma_0$  the associated labelled full simplicial graph, as defined in the introduction. We assume that the
associated Artin-Tits group is of large type;
recall that this means that any edge label is either infinite or at least $3$. 
In order to prove the result, we start  with the crucial notion of a 
\emph{critical word} introduced in  \cite{HoR}.  As before, by $S$ we denote the set of vertices of $\Gamma$ 
\subsection{Preliminaries}
\label{sec:prelim}
We start with some terminology. Let $S$ be a finite set. We denote by $(S\cup S^{-1})^*$ the set of words  on $S\cup S^{-1}$. A word $w$ on~$S\cup S^{-1}$ will be called \emph{positive} if it is written over $S$. It will be called \emph{negative}  if it is written over $S^{-1}$.
A \emph{signed} word is a word that is either positive or negative.
Some words are neither positive nor negative; when this is the case, we say
that the word $w$ is \emph{unsigned}.  We denote by $w^{-1}$ the word  that is the formal inverse of $w$. 
If $G$ is a group generated by $S$, a element $g$ of $G$ will be called \emph{positive} (relative to $S$) when it possesses a positive representative word
over $S$. 

A word $v$ on~$S\cup S^{-1}$  will be called a \emph{factor} of a word $w$ if the word $w$ can be decomposed with $w = w_1vw_2$ for some words $w_1,w_2$. 

 A word $w$ is called \emph{freely reduced} if it does not contain a factor of the form either $ss^{-1}$ or $s^{-1}s$ for some $s$ in $S$. Equivalently, $w$ is freely reduced if no operation of type $0$ can be applied. 

When $x$ and $y$ are two words and $m$ is a positive integer, we denote  by $[x,y,j\rangle$ the word $xyxy\cdots$ obtained as the alternating concatenation of
the words $x$ and $y$ that concatenates $j$ words, the first being $x$.
Hence $[x,y, 5\rangle = x[y,x,4\rangle =xyxyx$.
Similarly we denote by $\langle y,x,j]$ the word $\cdots yxyx$ obtained as the alternating concatenation of the words $x$ and $y$,  that concatenates $j$ words,
the last being $x$. So   $\langle y,x,6] = yxyxyx$.
In particular $[x,y,j\rangle = \langle y,x,j]$ or $[x,y,j\rangle =  \langle x,y,j]$
depending on whether $j$ is odd or even.
For instance, $[x,y, 5\rangle =xyxyx = \langle y,x,5]$ but $[x,y, 4\rangle =xyxy= \langle x,y,4]$. 
   
Finally, we say that a positive word is \emph{square-free}  if it does not contain a factor of the form $ss$ for any $s$ in $S$.   Now, for a word $w$,
we denote by $\tilde{p}(w)$ the length of the longest square-free positive
factor of $w$. Similarly we denote by $\tilde{n}(w)$ the length of the longest
negative factor of $w$ whose formal inverse is square-free.

\subsection{Dihedral Artin-Tits groups}
In the case where $\Gamma$ has only~$2$ vertices~$s,t$, the associated Artin
group~$A(\Gamma_0)$ is called a \emph{dihedral} Artin group. Indeed, in the
case where $m = m(\{s,t\})\neq\infty$, if one adds the relations $s^2 = 1$
and $t^2 = 1$ to the presentation of $A(\Gamma_0)$, then one obtains the
associated Coxeter group. On the other hand the obtained presentation is the one of a dihedral group of order $2m$. In particular the latter group is finite.  Following \cite{MaM}, for a word $w$ we set $p(w) = \min(\tilde{p}(w),m)$ and $n(w) = \min(\tilde{n}(w),m)$. 

In order to state the next result, let us introduce a restricted version of the operation of type~2.  
 \begin{defi}
\label{D:Special3}
For a finite presentation of a group~$\langle S, R\rangle$, we define
\emph{special transformations} of type~$2^\star$ on words on $S\cup S^{-1}$ to
be special transformations of type~2 for which, following the notation of Definition~\ref{D:Special2}, \\  
- in type $2r$ with $w = w_1v^{-1}v'w_2$, the word  $v$ is the greatest left common factor of the words  $vu$ and  $vw_1^{-1}$ and  the word  $v'$ is the greatest left common factor of the words  $v'u'$ and  $v'w_2$;\\
- in type $2\ell$ with $w = w_1v{v'}^{-1}w_2$, the word  $v$ is the greatest right common factor of the words  $uv$ and  $w_1v$ and  the word  $v'$ is the greatest right common factor of the words  $u'v'$ and  $w_2^{-1}v'$. \end{defi} 

\begin{exam} Consider  a \emph{dihedral} Artin group with $S = \{s,t\}$ and $m_{s,t} = 4$.  Consider the word $t^2sts^{-1}t$.   Then, we have $t^2{\bf sts^{-1}}t \goesR{2} t^2{\bf t^{-1}s^{-1}tst}t$,
but this transformation is not of type $2^\star$ because $v = st$; $uv = stst$ and $w_1v = t^{-1} tst$.
So  the word $tst$ is a greater common right divisor of $uv$ and $w_1v$ than $v$.
However,   $t{\bf tsts^{-1}}t \goesR{2^\star} t{\bf s^{-1}tst}t$.
The reader should note that in both cases the last letter of $w_1$ is $t$,
but  in the first case $tv$ is square-free, whereas in the second case $tv$ is not.\\
 
\end{exam}

For the remainder of this section, we assume that $\Gamma_0$ has only~$2$
vertices~$s,t$ (so $S=\{s,t\}$), and set $m = m_{s,t}$. So the group~$A(\Gamma_0)$
is a dihedral Artin group.  We set  $$\Omega(\Gamma_0) = \left\{ w\in (S\cup S^{-1})^*\mid w \textrm{ is freely reduced with} \max(p(w),n(w)) < m\right\}$$ 
The reader should note that the definition of $\Omega(\Gamma_0)$ ensures that operations of types 0 or 1 cannot be applied to words in $\Omega(\Gamma_0)$; only  operations of 
type~$2$ can apply.
Moreover, if $w$ belongs to $\Omega(\Gamma_0)$, then $\tilde{p}(w) = p(w)$ and $\tilde{n}(w) = n(w)$. 

For a non-empty word  $w$ in $\Omega(\Gamma_0)$,  we define the \emph{$\star$-decomposition of} $w$  to be the unique
sequence $(v_1,\ldots, v_k)$ of non-trivial  words such that $w = v_1\cdots v_k$ and each factor $v_i$ is the maximal square-free signed left factor of the word $v_i\cdots v_k$. 

\begin{prop} \label{PropOmegaClosedprel}
Let $w$ be a non-empty word that belongs to $\Omega(\Gamma_0)$, and
denote by $(v_1,\cdots, v_k)$ the $\star$-decomposition of $w$.
Assume that the word $w'$ is obtained from $w$ by an operation of type $2^\star$.
Then, $w'$ belongs to $\Omega(\Gamma_0)$ too. Moreover,  there exists a
subscript $i$ such that the $\star$-decomposition of $w'$ is
$(v_1,\cdots,v_{i-1},u_i,u_{i+1}, v_{i+2},\cdots,v_k)$  and the $2^\star$-operation  transforms $w$ into $w'$  by replacing the factor $v_iv_{i+1}$ of $w$ by the word $u_iu_{i+1}$. 
\end{prop}   

\begin{proof} 
Since the Artin group is assumed dihedral, the set $R$ consists of
a single relation $[s,t,m\rangle = [t,s,m\rangle$.
Write  $w = w_1w_2w_3$ and  $w' = w_1w_2'w_3$, where $w'_2$ is the factor
of $w'$ obtained from  the factor $w_2$ of $w$ by applying an operation of type $2^\star$.
We have either $w_2 = v^{-1}v'$ or $w_2 = vv'^{-1}$. 
In the first case, $w'_2 = u{u'}^{-1}$ where $vu = v'u'$ is the unique relation in $R$.
In the second case, $w'_2 = u^{-1}{u'}$ where $uv = u'v'$ is the unique relation in $R$.  

Assume the first case. The arguments for the second case are similar and left to the reader. 
We can assume without restriction that $v = [s,t,j\rangle$  and $v' = [t,s,\ell\rangle$ for some $j,\ell < m$. 
Then $vu  = [s,t,m\rangle$ and $v'u' = [t,s,m\rangle$.
Since the words $v_1,\cdots, v_k$ are signed words and any word $v_iv_{i+1}$ is either unsigned or not square-free,
there has to exist a subscript $i$ so that $v^{-1}$ is a right factor of $v_i$ and $v'$ is a left factor of $v_{i+1}$.
Now the operation that transforms $w$ into $w'$ is not just of type~$2$
but actually of type $2^\star$.
This imposes $v^{-1} = v_i$ and $v' = v_{i+1}$.  For $a$ in $S = \{s,t\}$,
denote by $\hat{a}$ the other letter of $S$.
Denote by $a$ and $a'$ the last letter of $v$ and  $v'$, respectively. 
Then, the first letter of $u$  is $\hat{a}$ and the last letter of ${u'}^{-1}$ is $(\hat{a'})^{-1}$.
Denote by $c$ the last letter $c$ of $w_1$, that is, the last letter of $v_{i-1}$,
and by $d$ the first letter of $w_3$, that is, the first letter of $v_{i+2}$.

The only cases where $w'$ could not be freely reduced would be if either $c = \hat{a}^{-1}$
or $d = \hat{a'}$. If c were equal to $\hat{a}^{-1}$ then we would have
$vc^{-1} = [s,t,j+1\rangle$ with $j+1 \leq n(w) < m$. so, $vc^{-1}$, that is,
$v\hat{a}$, would be a common left factor of $vw_1^{-1}$ and $vu$, contradicting the fact that the transformation is of type $2^\star$.
So  $c\neq \hat{a}^{-1}$. Similarly, $d\neq \hat{a'}$. Hence, $w'$ is a freely
reduced word.

We prove now that $\max(p(w'),n(w'))< n$. This will imply that
$w'$ belongs to $\Omega(\Gamma_0)$. Assume that one of the words 
$[s,t,m\rangle$ and  $[s,t,m\rangle$ or their inverses is a factor of $w'$.
Since $w$ belongs to $\Omega(\Gamma_0)$, this word cannot be a factor of
$w_1$ or of $w_3$. On the other hand, it can not be a factor of $w'_2$.
Then that word has to either overlap $w_1$ and $w'_2$, or to overlap $w'_2$ and $w_3$.
In the first case, it has to contain the word $c\hat{a}$ as a factor, which implies that $c=a$. In the second case it has to contain the word  ${d}^{-1}\hat{a'}$  which implies that $d=a'^{-1}$. Either of
these would contradict that fact that $w$ is freely reduced.
We deduce that 
$\max(p(w'),n(w')) < m $, and hence $w'$ belongs to $\Omega(\Gamma_0)$.

From $c \neq a$ and $d \neq a'^{-1}$,  we further deduce that
neither $v_{i-1}\hat{a}$ nor ${u'}^{-1}d$ can be a square-free signed word.
Hence 
$(v_1,\cdots,v_{i-1},u,{u'}^{-1}, v_{i+2},\cdots,v_k)$ is the $\star$-decomposition of $w'$.
\end{proof}

From Proposition~\ref{PropOmegaClosedprel} we deduce that
\begin{coro} \label{PropOmegaClosed}Assume that~$A(\Gamma_0)$ is a dihedral Artin group.
 Then,  $\Omega(\Gamma_0)$ is closed under the operations of type $2^\star$.
Moreover, if $w,w'$ belong to $\Omega(\Gamma_0)$ then $$w \goesR{2^\star} w' \iff w' \goesR{2^\star} w.$$
\end{coro}   

\subsection{Normal decomposition}

In the previous section we have seen that, in the case of a dihedral Artin-Tits group, the set $\Omega(\Gamma_0)$ is closed under operations of type $2^\star$. Here we study the connected components of $\Omega(\Gamma_0)$ under the $2^\star$ operation.
\begin{prop} \label{PropOmegaconnected}Assume that~$A(\Gamma_0)$ is a dihedral Artin-Tits group. 
Assume that $w$ and $w'$ belong to $\Omega(\Gamma_0)$, then, $$w \goesR{2^\star} w' \iff \textrm{$w$ and $w'$ represent the same element in}~A(\Gamma_0).$$
\end{prop}

To prove this result we use the fact that dihedral Artin-Tits groups are contained in the more general family of spherical type Artin-Tits groups. 
We recall that an Artin-Tits group is of spherical type when its associated
Coxeter group is finite, which is the case for dihedral Artin-Tits groups with $m<\infty$
as explained in the previous section.
The crucial argument is the following.

\begin{prop}\label{leftnormaldecomp}\cite{Cha} Let~$A(\Gamma_0)$ be a spherical type Artin
group, and let $g \in A(\Gamma_0)$.  Then there exists a unique
pair of positive elements $(g_1,g_2)$ such that (1) $g = g_1g_2^{-1}$ and (2)
whenever $h_1,h_2$ are positive elements with $g = h_1h_2^{-1}$, then there exists a positive element $h$ with $h_1= g_1h$ and $h_2= g_2h$.  
\end{prop}

The above decomposition $g_1g_2^{-1}$  is called the left normal decomposition
of  $g$. 
The following is an immediate consequence of the above proposition.    
\begin{prop}\label{PropCaracNormForm} Assume that~$A(\Gamma_0)$ is a spherical
type Artin group. Let $g$ belong to~$A(\Gamma_0)$.  Assume that $(w_1,w_2)$
is a pair of positive words so that $w_1w_2^{-1}$ is a word representative of~$g$. Denote by $g_1$ and $g_2$ the (positive) elements of~$A(\Gamma_0)$ represented by $w_1$ and $w_2$ respectively. Then, the following are equivalent:
\begin{enumerate}
\item $g_1g_2^{-1}$  is not the left normal decomposition of  $g$,
\item there exist two distinct words $w$ and $w'$ so that  $w_1w_2^{-1}\goesR{1} w\goesR{0} w'$. 
\end{enumerate}
\end{prop}

Using this property, we can deduce the following.
\begin{prop} \label{PropConvNF} Let~$A(\Gamma_0)$ be a dihedral Artin-Tits
group, and suppose that $w \in \Omega(\Gamma_0)$. Denote by $g$ the element
of~$A(\Gamma_0)$ represented by $g$ and by  $(g_1,g_2)$ the pair of positive
elements of~$A(\Gamma_0)$ corresponding to the  
left normal decomposition of  $g$. 

Then,  there exists  a pair of positive words $(w_1,w_2)$  so that the words $w_1$ and $w_2$ represent $g_1$ and $g_2$, respectively, and $$w \goesR{2^\star} w_1w_2^{-1}.$$ \end{prop}

\begin{proof} 
Let $(v_1,\ldots,,v_k)$ be the $\star$-decomposition of $w$. Say that the pair $(i,j)$, with $i<j$, is an inversion of the word $w$ if $v_i$ is a negative word whereas $v_j$  is a positive word. Denote by $\textrm{Inv}(w)$ the number of inversions of the word $w$. 

Assume first $\textrm{Inv}(w) =0$. Then there exists an  index $i$ such that $v_1\cdots v_i$ is a positive (or empty) word and  $v_{i+1}\cdots v_k$ is negative (or empty). 
Since $w \in \Omega(\Gamma_0)$, it follows
(immediately from the definition of $\Omega(\Gamma_0)$) that no operation of
type~$1$ or $0$ can be applied to $w$. Hence, we deduce from
Proposition~\ref{PropCaracNormForm} that the words $v_1\cdots v_i$ and
$(v_{i+1}\cdots v_k)^{-1}$ represent $g_1$ and $g_2$, respectively.   

So now suppose that $\textrm{Inv}(w)\neq 0$. Then there exists an  index
$i$ such that $v_i$ is a negative word whereas $v_{i+1}$  is positive.
By definition of the decomposition $v_1\cdots v_k$, an operation of type
$2^\star$ can be applied to $w$ to replace the word $v_iv_{i+1}$ by the
word $u_iu_{i+1}$ for which $ v_{i+1}u^{-1}_{i+1}=  v_i^{-1}u_i$ is the unique
relation in the presentation of the group~$A(\Gamma_0)$. We obtain the word
 $w' = v_1\cdots v_{i-1}u_iu_{i+1}v_{i+2} \cdots v_k$ . By Proposition~\ref{PropOmegaClosedprel}, the word $w'$ remains in $\Omega(\Gamma_0)$  and $(v_1,\ldots, v_{i-1},u_i,u_{i+1},v_{i+2},\ldots, v_k)$ is the $\star$-decomposition of $w'$.
It is immediate that $\textrm{Inv}(w') = \textrm{Inv}(w)-1$.
Repeating the argument, we deduce there exists a word $w''$ in $\Omega(\Gamma_0)$ such that  $w \goesR{2^\star} w''$ and  $\textrm{Inv}(w'') = 0$. 
Denote by $(u_1,\ldots, u_k)$ the $\star$-decomposition of the obtained word $w''$. By the first case,  there exists an  index $i$ 
such that  the words $u_1\cdots u_i$ and $(u_{i+1}\cdots u_k)^{-1}$ represent $g_1$ and $g_2$, respectively.  This proves the proposition. \end{proof}

We can now prove Proposition~\ref{PropOmegaconnected}.
\begin{proof}[Proof of Proposition~\ref{PropOmegaconnected}]
Assume that $w$ and $w'$ belong to $\Omega(\Gamma_0)$.

Certainly if $w \goesR{2^\star} w'$, then $w$ and $w'$ represent the same
element in $A(\Gamma_0)$.

Conversely, assume $w$ and $w'$ represent the same element $g$, and denote
its  left normal decomposition by $g_1g_2^{-1}$.  By Proposition~\ref{PropConvNF}, there exist positive words $w_1,w'_1,w_2,w'_2$ so that (a) $w_1$ and $w'_1$ represent $g_1$; (b) $w_2$ and $w'_2$ represent $g_2$; (c) $w \goesR{2^\star} w_1w_2^{-1}$ and $w' \goesR{2^\star} w'_1{w'}^{-1}_2$.

From (c) and Corollary~\ref{PropOmegaClosed} we have $ w'_1{w'}^{-1}_2\goesR{2^\star} w'$.
By (a), (b) and \cite{Par}, we have $w_1\goesR{1}w'_1$ and $w_2\goesR{1}w'_2$. 
Since $w$ and $w'$ belong to $\Omega(\Gamma_0)$, the words  $w_1w_2^{-1}$ and $w_1w_2^{-1}$ must too,
by Corollary~\ref{PropOmegaClosed}, and hence
no non-trivial operations of type~$1$ can be applied to them or to their factors. Hence, $w_1 = w'_1$ and $w_2 = w'_2$.  It follows that $w \goesR{2^\star} w_1w_2^{-1} \goesR{2^\star} w'$.
\end{proof}

\subsection{Critical words}\label{seccritword}
For all of this section, we assume that~$A(\Gamma_0)$ is a dihedral Artin group
with $S = \{s,t\}$ and $m = m_{s,t} \neq\infty$.
Set $\Delta = \overline{[s,t, m\rangle} =  \overline{[t,s, m\rangle}$. Then,
$\Delta   \overline{s}^{\pm 1}=  \overline{\delta(s)}^{\pm 1}\Delta$ and $\Delta  \overline{t}^{\pm 1} =  \overline{\delta(t)}^{\pm 1}\Delta $,
where $\delta$ is the permutation of $S$ defined by $\delta(s) = \overline{s^\Delta}$,$\delta(t)=\overline{t^\Delta}$, acting as the identity when $m$ is even, and swapping $s$ and $t$ when $m$ is odd.
The map $\delta$ naturally extends to a involution $\delta$ on
$(S\cup S^{-1})^*$ such that for every word $v$ on $(S\cup S^{-1})^*$ one has $\Delta \overline{v} = \overline{\delta(v)}\Delta$. 
 
We recall that for a group $G$ generated by a set $S$,  a \emph{geodesic word
representative} of an element $g$ is a representative word of $g$ over
$S\cup S^{-1}$  that is of minimal length amongst the representative words
of $g$. Given an arbitrary Artin-Tits group, it is an open question to
obtain an algorithm that for each word returns a geodesic word representing the same element.
However, in the case of dihedral Artin-Tits groups, and in the more general case of Artin-Tits groups of large type, a positive answer has been obtained in~\cite{MaM} and in~\cite{HoR}, respectively.
  
More precisely, it is proved in~\cite{MaM}, in a dihedral group $A_S$,  a freely
reduced word~$w$ over $s,t$ is geodesic if and only if  $p(w)+n(w)\leq m$.  In addition in~\cite[Section 2]{HoR}, the authors introduced the notion of \emph{critical word}. A critical word may be signed or unsigned; for now we restrict to unsigned critical words.
\begin{defi} Assume~$A(\Gamma_0)$ is a dihedral Artin group. A freely
reduced unsigned geodesic word $w$ with $p(w)+q(w) = m$ is called an \emph{unsigned  critical word} if   either $$w = [x,y, p\rangle\ w'\  [z,t, q\rangle^{-1}$$  or $$w = [x,y, q\rangle^{-1}\ w'\  [z,t, p\rangle$$ with $\{x,y\} = \{z,t\} = S$, $p = p(w)$ and $q = q(w)$. \end{defi}

Obviously the conditions $p(w)=p,n(w)=n$ impose some restrictions
on the word $w'$.

Recall that $[x,y, p\rangle = \langle x',y', p]$ where $\{x,y\} = \{x',y'\}$,
and  that $[z,t, q\rangle = \langle z',t', q]$ where $\{z,t\} = \{z',t'\}$.
Now, let  $w = [x,y, p\rangle\ w'\  [z,t, q\rangle^{-1}$ be a word over $(S\cup S^{-1})$ such that $p+q = m$.
One has 
$$\displaylines{\overline{w} = \overline {[x,y, p\rangle\, w'\,  [z,t, q\rangle^{-1}} 
= \overline{\langle x,y, q]^{-1}\,\, \Delta \, w'\,  [z,t, q\rangle^{-1}}
= \overline{\langle x,y, q]^{-1}\,\, \delta(w')\Delta\,  [z,t, q\rangle^{-1}} \cr\hfill\cr
= \overline{\langle x,y, q]^{-1}\, \delta(w')\,  \langle z,t,p]} }$$  
Set 
$$\tau\left([x,y, p\rangle\ w'\  [z,t, q\rangle^{-1}\right) 
= \langle x,y, q]^{-1}\, \delta(w')\,  \langle z,t,p];$$  
then one has $\overline{\tau(w)} = \overline{w}$. 
So, $w$ is unsigned and geodesic if and only if $\tau(w)$ is. 
Moreover,  $w$ is an unsigned  critical word with $p= p(w)$ and $q = q(w)$ 
if and only if $\tau(w)$ is an unsigned  critical word  with $p= p(\tau(w))$ 
and $q = q(\tau(w))$. 
In this case, we also set $$\tau \left(\langle x,y, q]^{-1}\, \delta(w')\,  \langle z,t,p]\right) = [x,y, p\rangle\ w'\  [z,t, q\rangle^{-1}.$$ 
Then $\tau$ is an involution over the set of unsigned critical words. 

The above is proved, together with other properties of critical words
in \cite[Proposition 2.1]{HoR}. From those properties
we shall  also need the following, in Section~\ref{secTransv}:

\begin{lemm} \label{Lemtauetgen}Assume that~$A(\Gamma_0)$ is a dihedral Artin group with generating set $S = \{s,t\}$. 
Let $w$ be an unsigned critical word. The rightmost letter of $w$ is $s^{\pm 1}$ if and only if the rightmost letter of $\tau(w)$ is $t^{\mp 1}$. 
\end{lemm}

\begin{lemm} \label{Lemtauet2star}Assume~$A(\Gamma_0)$ is a dihedral Artin group with generating set $S = \{s,t\}$. Let $w$ be an unsigned critical word.
Then $w \goesR{2^\star} \tau(w)$.
\end{lemm}
\begin{proof} The words $w$ and $\tau(w)$ are unsigned critical words. In
particular $p(w)+q(w) = m$ with $p(w)\neq 0$ and $q(w)\neq 0$.
So $w$ is a freely reduced word so that $p(w)< m$ and $q(w)<m$, and hence $w$ belongs to
$\Omega(\Gamma_0)$. But $\tau(w)$ is also a reduced word with $p(\tau(w)) = p(w)< m$ and $q(\tau(w))= q(w)<m$, and hence also $\tau(w)$ belongs to $\Omega(\Gamma_0)$, too. Since $\overline{\tau(w)} = \overline{w}$, Proposition~\ref{PropOmegaconnected} proves  that $w \goesR{2^\star} \tau(w)$.
\end{proof}

In \cite{HoR}, the authors also introduced a notion of a positive (\emph{resp.} negative) critical word. If~$A(\Gamma_0)$ is a dihedral Artin group with generating set $S = \{s,t\}$,  a positive word $w$ with $\tilde{p}(w)= p(w) =m$ is called critical if it has the
form $[x,y,m\rangle w'$  or $w'\langle m,x,y ]$ with $\{x,y\} = \{s,t\}$ and $p(w')<m$.
Similarly a negative word $w$ with $\tilde{n}(w)= n(w)=m$ is
called critical  if it has the form
form $[x^{-1},y^{-1},m\rangle w'$  or $w'\langle m,x^{-1},y^{-1} ]$, with $\{x,y\} = \{s,t\}$ and $n(w')<m$. 
Further, the definition of $\tau$ can be extended to that of an involution
on critical words that maps
each positive, resp. negative critical word to another positive resp. negative
critical word representing the
same group element. 
More precisely, if $w = [x,y,m\rangle w'$ is a positive critical word, 
then $\overline{[x,y,m\rangle w'} = \overline{\delta(w')\langle x,y,m]} 
= \overline{\delta(w')\langle y,x,m]}$. 
 Only one of the two words $\delta(w')\langle x,y,m]$ and $\delta(w')\langle y,x,m]$ has different last letter from $w$.
If $w'$ is not empty it is also the only one of these two words that is critical. We define $\tau(w)$ to be this
word. We define $\tau(w)$ for a negative critical word $w=w' \langle x,y,m]$
in a similar way. It is easy to see that $\tau$ remains an involution.

It is straightforward to prove the following, and so we omit the details:
\begin{lemm}\label{Lemtauet1}  Assume that~$A(\Gamma_0)$ is a dihedral Artin group with generating set $S = \{s,t\}$.  For any positive or negative critical word $w$,
$w \goesR{1} \tau(w)$.   
\end{lemm}

 We now define the special transformation $\beta$ in the  case of dihedral Artin-Tits groups: 
\begin{defi} \label{D:Specialbeta}
Assume that~$A(\Gamma_0)$ is a dihedral Artin group with generating set $S = \{s,t\}$. Let $\tau$ be the above involution on the set of critical words.  We define \emph{special transformations} of type~$\beta$ on words on $S\cup S^{-1}$ as follows:\\
\emph{type $\beta$} : Replace some factor $w$ with~$\tau(w) $, where $w$ is a  critical word.\\
\end{defi}
From this definition and Lemmas~\ref{Lemtauet2star} and~\ref{Lemtauet1}, it follows:
\begin{lemm}\label{Lemtauet2}  Assume that~$A(\Gamma_0)$ is a dihedral Artin group with generating set $S = \{s,t\}$. For any two words $w, w'$, one has 
$$w \goesR{\beta} w'\implies w \goesR{1,2^\star} w'.$$   
\end{lemm}
\begin{prop}\cite[Sec.~2]{HoR}\label{PropHoR1}   Assume that~$A(\Gamma_0)$ is
a dihedral Artin group.  Let $w$ be a word over $S\cup S^{-1}$. Then there
exists a
geodesic word $w'$  such that  $\overline{w} = \overline{w'}$ and $$w \goesR{0,\beta} w'.$$  
\end{prop} 
\begin{proof} 
The result follows immediately from \cite[Lemma 2.3]{HoR},
from which we deduce that if $w$ is freely reduced  but non-geodesic then
it must possess a
critical word $w_1$ as a factor, and then the word $w'$ obtained from $w$ by
replacing the factor $w_1$ with $\tau(w_1)$ is not freely reduced. So a
combination of free reduction and at most one type $\beta$ transformation reduces $w$ to $w'$.
\end{proof}
Combining Lemma~\ref{Lemtauet2} and Proposition~\ref{PropHoR1} we get the following corollary.
\begin{coro}\label{coroHstardih}
Assume that~$A(\Gamma_0)$ is a dihedral Artin group with $S$ as generating set.  Let $w$ be a word over $S\cup S^{-1}$. There exists a geodesic word $w'$  such that  $\overline{w} = \overline{w'}$ and $$w \goesR{0,1,2^\star} w'.$$   
\end{coro}

Transformations of type $2^\star$ are special cases of transformations of type~$2$ and the unique geodesic representative word of $1$ is the empty word $\varepsilon$. Therefore as a consequence of the above corollary,  we recover :
\begin{coro}\cite{DeG} \label{PropHDih}Assume~$A(\Gamma_0)$ is a dihedral Artin group with $S$ as generating set.  Let $w$ be a word over $S\cup S^{-1}$.  Then, $$w \gozotR \varepsilon \iff \overline{w} = 1$$ 
\end{coro}
\subsection{The case of large type Artin-Tits groups}
Our objective here is to prove the following result, which is a special case of Theorem~\ref{TH1}.  
\begin{prop}\ \label{PropHLarge2}Assume~$A(\Gamma_0)$ is an Artin group of large type with $S$ as generating set.  Let $w$ be a word over $S\cup S^{-1}$. There exists a geodesic word $w'$  such that  $\overline{w} = \overline{w'}$ and $$w \goesR{0,1,2^\star} w'.$$ 
\end{prop}
The proof is similar to the proof of Corollary~\ref{coroHstardih};
it follows from a combination of Lemma~\ref{Lemtauet3} below (a  straightforward extension of Lemma~\ref{Lemtauet2}) and results from \cite{HoR}, as we now describe.

Assume that~$A(\Gamma_0)$ is a Artin group with $S$ as generating set.
By \cite{VdL,Par}, the subgroup generated by any subset $T$ of $S$  is  canonically isomorphic to the Artin-Tits group associated with the full
subgraph of $\Gamma_0$. These subgroups are called the \emph{parabolic}
subgroups of~$A(\Gamma_0)$.  So, if $T = \{s,t\}$  with $m_{s,t}\neq \infty$
then the parabolic subgroup generated by $T$ is a dihedral Artin group . We  can
define the family of (signed or unsigned) critical words of~$A(\Gamma_0)$ as
the (disjoint) union of the sets of  (signed or unsigned) critical words of all dihedral parabolic subgroups of~$A(\Gamma_0)$. The involution $\tau$ is still
well defined on this set, as are the associated operations of type $\beta$ on words over $S\cup S^{-1}$. So, the statement of Lemma~\ref{Lemtauet2} extends
to the framework of Artin groups of large type.

\begin{lemm}\label{Lemtauet3}  Assume that~$A(\Gamma_0)$ is an Artin group of
large type. For any two words $w, w'$, one has 
$$w \goesR{\beta} w'\implies w \goesR{1,2^\star} w'.$$   
\end{lemm}
 
In the next section, we shall need the notion of rightward
critical sequences from \cite{HoR}. But we postpone the explicit definition of
these to Section~\ref{secCritSequ}. For now, all
we need is the following associated result, derived from \cite[Proposition 3.3]{HoR} 
\begin{lemm}
Assume that ~$A(\Gamma_0)$ is a Artin group of large type
with $S$ as generating set.  Let $w$ be a word over $S\cup S^{-1}$.  If the
word $w$ is not geodesic then it can be transformed into a word that is not
freely reduced by a sequence of special transformations of type $\beta$. 
\end{lemm}   
With Lemma~\ref{Lemtauet3} at hand, Corollary~\ref{CoroHLarge2} follows immediately.
\begin{coro}\ \label{CoroHLarge2}Assume that~$A(\Gamma_0)$ is a Artin group of
large type with $S$ as generating set.  Let $w$ be a word over $S\cup S^{-1}$.  Then, $$w \goesR{0,1,2^\star} \varepsilon \iff \overline{w} = 1$$ 
\end{coro}

Thus Property~$H$ is verified by Artin-Tits groups of large type.
In fact we can prove a stronger property than $H$, property 
$H\!^{\textit{+}}$ below,
using the
convexity of parabolic subgroups of Artin groups (proved
for large type Artin-Tits groups in \cite{HoR}, but for all Artin groups
in \cite{ChP}):
\begin{prop}\cite{ChP}  \label{prop:ChP}
Assume that~$A(\Gamma_0)$ is an Artin-Tits group with $S$ as generating set.  
Let $S_0$ be included in~$S$. If $g$ belongs to the standard parabolic 
subgroup $A_{S_0}$ then all its geodesic representative words over $S\cup S^{-1}$ are actually over $S_0\cup S_0^{-1}$.   
\end{prop}

\begin{defi}
Assume that $(S, R)$ is an Artin-Tits presentation.  We say that $(S, R)$ satisfies 
\emph{Property~$H\!^{\textit{+}}$} if for all~$S_0 \subseteq S$ and any word $w$ over $S\cup S^{-1}$
 that represents an element of the standard parabolic subgroup $A_{S_0}$,
there exists a word $w_0$ over $S_0\cup S_0^{-1}$ so that $$w \goesR{0,1,2^\star} w_0$$ 
\end{defi}
\begin{coro}\ \label{CoroHplusLarge} Every Artin-Tits presentation of large type satisfies Property~$H\!^{\textit{+}}$.
\end{coro}
\begin{proof}  By  Proposition~\ref{PropHLarge2} there exists a geodesic word $w_0$ so that  $w \goesR{0,1,2^\star} w_0$. But $w_0$ is a geodesic word and $\overline{w_0}$ belongs to the standard~$A_{S_0}$. So $w_0$ is over $S_0\cup S_0^{-1}$. 
\end{proof}

We note that we derive 
Property~$H\!^{\textit{+}}$ not simply from Property $H$ plus convexity but
rather from convexity together with the technical results from which Property
$H$ is derived.
Hence it does not follow simply from convexity that Artin-Tits presentations
of spherical type or of FC type satisfy 
Property~$H\!^{\textit{+}}$ as well as $H$. But nonetheless the property must 
hold for both (and also for sufficiently large Artin-Tits groups), since it is implied by the Property $H^\#$ that is introduced
in  Section~\ref{sec:sufflarge};
see the comment following Proposition~\ref{propHDpAmal}.

\section{Sufficiently large Artin-Tits groups}
\label{sec:sufflarge}
Our objective is now to prove Property $H$ for the family of sufficiently large Artin-Tits groups, which we define in this section. Our strategy is similar
to the one used in \cite{DeG}  to prove property $H$ for Artin-Tits groups
of FC type. Namely we prove that a stronger property~$H^\#$ holds for large
type Artin groups and then extend it to sufficiently large Artin-Tits groups using amalgamation. 
\subsection{A characterization of Sufficiently large Artin-Tits groups}
\label{sec:char_sufflarge}
Sufficiently large Artin-Tits groups were introduced in \cite{HoR2}, with the
definition expressed in
terms of the Coxeter graph. Here we obtain a characterization that proves
that these groups are built out of large type Artin-Tits groups and finite 
rank free abelian groups.  We first recall the original definition of sufficiently large Artin-Tits groups. 
\begin{defi}\label{def:sufflarge} \cite{HoR2} Assume that~$A(\Gamma_0)$ is an Artin group with
Coxeter graph $\Gamma$. We say that ~$A(\Gamma_0)$ is sufficiently large if the full subgraph on any triple of distinct
 vertices of  $\Gamma$ is either completely disconnected, or of large type, or contains an edge with a label $\infty$. In other words, if $s,t,u$ are three
distinct vertices with $m_{s,t} = 2$ and  $m_{s,u} \neq 2,\infty$, then  $m_{t,u} = \infty$.  \end{defi}

From this definition we deduce that 
\begin{prop} \label{propcarSLAT}The family of sufficiently large Artin-Tits groups is the smallest family of Artin-Tits groups closed under amalgamation over a standard parabolic subgroup, and which contains finitely generated free Abelian Artin-Tits groups and large type Artin groups.   
\end{prop}
\begin{proof} First, it is immediate from the definition that every standard parabolic subgroup of a sufficiently large Artin-Tits group is also sufficiently large.

Now,  assume that~$A_S$ is a  sufficiently large Artin-Tits group, and
denote by $S$ its generating set. If there is no $\infty$-labelled edge, then  either $m_{s,t} = 2$ for all pairs $s,t$ in $S$, or there is no pair $s,t$  in $S$ so that $m_{s,t} =2$.
So~$A_S$  is either free abelian or of large type. Assume there exist $s_1,s_2$  in $S$  so that $m_{s_1,s_2} =\infty$. Then, it is well-known
that~$A_S$  is canonically isomorphic to the amalgamated product $A_1*_{A_{1,2}}A_2$ where $A_1$, $A_2$ and $A_{1,2}$ are the standard parabolic subgroups of $A_S$ generated by $S\setminus\{s_1\}$, $S\setminus\{s_2\}$ and $S\setminus\{s_1,s_2\}$.
This proves by induction on the number of  $\infty$-labelled edges that sufficiently large Artin-Tits groups belong to  the smallest family of Artin-Tits groups closed under amalgamation over a standard parabolic subgroup and which
contains finitely generated free Abelian groups and Artin-Tits groups of large type.
Conversely, assume that $A_1$, $A_2$ are two Artin-Tits groups ,
generated by $S_1$ and $S_2$ respectively ,with a common standard parabolic subgroup $A_{1,2}$ generated by $S_{1,2} = S_1\cap S_2$. 
Denote by $\Gamma_1$ and $\Gamma_2$ the Coxeter graphs of $A_1$, $A_2$, respectively. Then the group $A_1*_{A_{1,2}}A_2$ is the Artin-Tits group whose graph
is obtained by glueing $\Gamma_1$ and $\Gamma_2$ along the  common subgraph
generated by  $S_{1,2}$, and adding an $\infty$-labelled edge between each pair
$s,t$ of vertices so that $s$ lies in $S_1\setminus S_{1,2}$ and $t$ lies in $S_2\setminus S_{1,2}$.
Therefore, if $A_1$, $A_2$  are sufficiently large Artin-Tits groups, so is $A_1*_{A_{1,2}}A_2$.    \end{proof}

\subsection{The property~$H\!^\textit{+}$}
\label{sec:Hplus}
From the previous section, sufficiently large Artin-Tits groups are built out
of large type Artin-Tits groups and free abelian groups through iterated sequences
of amalgamations.
So it would be natural to try and prove that  sufficiently large Artin-Tits groups satisfy
Property $H$ via a proof that when two group presentations satisfy Property $H$ and can be amalgamated in a proper way then the obtained group
presentation also satisfies Property $H$; unfortunately, there is no 
clear reason why this should be true.
However, since
Property ~$H$ is implied by
Property~$H\!^\textit{+}$ (by considering $S' = \emptyset$), 
a straightforward argument 
proves that the amalgamated product over a standard parabolic subgroup
of two Artin groups satisfying   
Property~$H\!^\textit{+}$ 
satisfies property $H$, as follows: 
 
\begin{prop}\label{propHplus_amalg} Let $(S,R)$ be an Artin-Tits presentation. 
Assume $S = S_1\cup S_2$. Set $S_{1,2} = S_1\cap S_2$. Denote by $A_1$, $A_2$ and $A_{1,2}$ the parabolic subgroups of $A_S$ generated by $S_1$, $S_2$ and
$S_{1,2}$, respectively and by $(S_1,R_1)$, $(S_2,R_2)$ and $(S_{1,2},R_{1,2})$ their Artin-Tits presentations. 
Assume
that $A_S = A_{S_1}*_{A_{S_{1,2}}}A_{S_2}$ and  $(S_1,R_1)$, $(S_2,R_2)$ and $(S_{1,2},R_{1,2})$ 
satisfy Property~$H\!^\textit{+}$. Then the presentation~$(S,R)$ satisfies Property $H$. 
\end{prop}   
\begin{proof} 
Let $w$ be a non-empty word on $S\cup S^{-1}$.  Assume $\overline{w} = 1$.
We need to prove that $w \gozotR \varepsilon$.

We 
decompose $w$ as $w = w_1\cdots w_k$ for some minimal $k \geq 1$, such that 
each subword $w_i$ is written either over $S_1\cup S_1^{-1}$ 
or $S_2\cup S_2^{-1}$ (we call the $w_i$ the factors of $w$). 
We prove by induction on $k$ that $w \gozotR \varepsilon$.
Clearly the result is true when $k=1$.

If $k\geq 2$, it follows from consideration of amalgam normal forms (see, for example \cite[Corollary 1.8]{GodellePJM}) that some $\overline{w_i}$ belongs to 
$A_{S_{1,2}}$.
Then, where $w_i$ is written over $S_\ell \cup S_\ell^{-1}$, since
$(S_\ell,R_\ell)$ satisfies Property~$H\!^\textit{+}$,   there exists $w'_i$ in  $S_{1,2}\cup S_{1,2}^{-1}$ 
so that  $w_i \gozotR w'_i$. 
We let $w'$ be the word derived from $w$ by replacing $w_i$ by $w'_i$.
Then $w \gozotR w'$. 
If $1<i$ then $w'$ contains $w_{i-1}w'_i$ as a subword, and otherwise
$i<k$ and $w'$ contains $w'_iw_{i+1}$ as a subword; either subword is written over $S_j$. So now, since we have decomposed $w'$ as a product of at most $k-1$
factors, induction implies that $w' \gozotR \varepsilon$.
It follows that
$w \gozotR w' \gozotR \varepsilon$, and our proof is complete.
\end{proof}
Unfortunately this approach does not appear
to lead to a proof that all Artin-Tits groups of sufficiently large type
satisfy property~$H$, since it is not clear that the presentation of the amlgamated product 
would satisfy  Property~$H\!^\textit{+}$, as we would need in order to
prove that iterating the construction gives a group satisfying~$H$.
Instead we need a stronger property than $H$ that \emph{is} preserved under 
amalgamation. 

\subsection{The property~$H^\#$}
\label{sec:Hsharp}
The article \cite{DeG} introduces a stronger property than Property $H$, namely Property $H^\#$. And in that article it is proved that the
strategy proposed in Section~\ref{sec:Hplus} can be applied to  Property $H^\#$. This is the way that Property $H$
is proved to be satisfied by FC type Artin-Tits groups. We follow the same
approach here to prove that Property $H$ is satisfied by sufficiently large Artin-Tits groups.

 If $H$ is a subgroup of a group~$G$, we call an \emph{$H$-transversal}  of $G$ any  subset of~$G$ that contains exactly one element from each left-$H$-coset,
and in particular contains~$1$. 
If $A_S$ is an Artin-Tits group generated by $S$, then we call an
\emph{$S$-sequence of transversals} any sequence $(T(S'))_{S'\subseteq S}$
for which the set $T(S')$ is a $A_{S'}$-transversal of $A_S$, for each $S' \subseteq S$. 
\begin{defi}
\label{D:Propp}
Assume that $(S, R)$ is an Artin-Tits presentation and $T$ is an associated $S$-sequence of transversals.
We say that $(S, R)$ satisfies \emph{Property~$H^\#$ (with respect to}~$T$) if,
for all~$S', S_0 \subseteq S$ and any word $w$ over $S\cup S^{-1}$, setting $S'_0 = S' \cap S_0$, the following relation is satisfied:
\begin{equation*}
\displaylines{\overline{w} \in A_{S'} \ \Rightarrow \hfill\cr\hfill \exists v \in (S'\cup {S'}^{-1})^*,\  \exists u \in (S'_0\cup {S'_0}^{-1})^* \ (w \gozotR v u \mbox{\ and \ }\overline{v} \in T(S_0)).}
\end{equation*}
\end{defi}

Roughly speaking, this property says that for every parabolic subgroup
$A_{S_0}$ of $A_S$,
there is an $A_{S_0}$-transversal $T(S_0)$ of $A_S$,
so that one is able to find the decomposition of any element $g$ as the
product of its transversal representative and an element of $A_{S_0}$.
Moreover, this can be done at the level of the word representatives, using
operations of type  $0,1,2$, only.  Finally  when the element $g$, itself,
belongs to another parabolic subgroup $A_{S'}$
then the two obtained representative words of the decomposition can be
written over  $S'\cup S'^{-1}$.
This imposes in particular the condition that the restriction of $T(S_0)$ to the elements
that belong to $A_{S'}$ provide an  $A_{S_0\cap S'}$-transversal of $A_{S'}$.
Considering the case where $S'$ is the empty set, we get the first part of the following proposition

\begin{prop}\cite[Proposition 2.14]{DeG}\label{propHDpAmal} Let $(S,R)$ be an Artin-Tits presentation. 
\begin{enumerate}
\item If $(S,R)$ satisfies Property $H^\#$, then $(S,R)$ satisfies Property $H$. 
\item Assume $S = S_1\cup S_2$. Set $S_{1,2} = S_1\cap S_2$.
Denote by $A_1$, $A_2$ and $A_{1,2}$ the parabolic subgroups of $A_S$ generated
by $S_1$, $S_2$ and $S_{1,2}$, respectively and by $(S_1,R_1)$, $(S_2,R_2)$
and $(S_{1,2},R_{1,2})$ their Artin-Tits presentations. Assume
that $A_S = A_{S_1}*_{A_{S_{1,2}}}A_{S_2}$ and  $(S_1,R_1)$, $(S_2,R_2)$ and $(S_{1,2},R_{1,2})$ satisfy Property $H^\#$. Then $(S,R)$ satisfies Property $H^\#$. 
\end{enumerate}
\end{prop}   
Note also that Property $H^\#$ immediately implies Property $H^{\textit{+}}$.
We shall use Propositions~\ref{propcarSLAT} and~\ref{propHDpAmal} in the final section
to deduce Theorems~\ref{TH1} and ~\ref{TH2} from Proposition~\ref{protrans},
which proves Property $H^\#$ for Artin-Tits groups of large type, together with
the same result for finite rank free Abelian groups.
Since finite rank Abelian groups are all of spherical type, and hence covered
by the results of \cite{DeG}, we focus now on the 
case of   large type Artin-Tits groups.

\subsection{Critical sequence}\label{secCritSequ}
In order to prove Property $H^\#$ for large type Artin-Tits groups, we have first to find a $S$-sequence of transversals.  The key argument turns out to
involve critical factorizations and associated critical sequences of transformations, which we have already mentioned,
and for which we now need a definition.
We recall the notion of a critical word from Section~\ref{seccritword}. 
Our definition of a critical factorization is taken  from \cite{HoR}.
Roughly speaking a critical factorization of a word is more-or-less 
its expression as a concatenation of
critical words with one letter overlaps. 

\begin{defi} \label{defcrtifact}Assume $A_S$ is an Artin-Tits group. 
Let $n$ be a positive number, and suppose that $w$ is freely reduced.
\begin{enumerate}
\item We define a {\em critical factorization} of $w$ of length 1 to be an
expression of $w$ as a concatenation $\alpha v_1 \beta$ of (possibly empty) words $\alpha,\beta$ and
a critical word $v_1$, and call $v_1$ the {\em critical factor} of that factorization. A critical factorization of length 1 is considered to be
both rightward and leftward.

\item For $n \geq 2$,  we define a {\em rightward critical factorization} of $w$ of
length $n$ with {\em critical factors} $v_1,\ldots,v_n$ to be an expression of $w$ as a concatenation
$\alpha v_1\cdots v_n \beta$, where
$v_1$ is critical, and, where $s$ is the last letter of $\tau(v_1)$ (as defined in Section~\ref{seccritword}),
the words $sv_2,v_3,\ldots,v_n$ are the critical factors of a critical factorization of
length $n-1$ of the word $sv_2v_3\cdots v_n$. 
\item For $n \geq 2$,  we define a {\em leftward critical factorization} of $w$ of
length $n$ with {\em critical factors} $v_1,\ldots,v_n$ to be an expression of $w$ as a concatenation
$\alpha v_n\cdots v_1 \beta$, where
$v_1$ is critical, and, where $s$ is the first letter of $\tau(v_1)$ ,
the words $v_2s,v_3,\ldots,v_n$ are the critical factors of a critical factorization of
length $n-1$ of the word $v_nv_{n-1}\cdots v_2s$. 
\end{enumerate}
\end{defi}
Note that a word $w$ might admit many critical factorizations, 
both rightward and leftward, and that we do not require the subwords $\alpha$, $\beta$ of $w$ to be non-empty.

It is convenient to extend our definition of $\tau$ so that it can be applied to
the sequence of critical factors in a critical factorization.
\begin{enumerate}
\item
We define the image 
$\tau(v_1,\ldots,v_n)$ of the sequence of critical factors of a rightward
critical factorization to be the $n$-tuple $(v'_1,v'_2,\ldots, v'_n)$. 
where $v'_1$ is the maximal proper prefix of $\tau(v_1)$, 
and $(v'_2,\ldots,v'_n)$ is the image under $\tau$ of the sequence 
of critical factors $(sv_2,\ldots,v_n)$ of $sv_2\cdots v_n$.
\item
We define the image 
$\tau(v_1,\ldots,v_n)$ of the sequence of critical factors of a leftward critical factorization
to be the $n$-tuple $(v`_1,v'_2,\ldots,v'_n)$. 
where $v'_1$ is the maximal proper suffix of $\tau(v_1)$, 
and $(v'_2,\ldots,v'_n)$ is the image under $\tau$ of the sequence 
of critical factors $(v_2s,\ldots,v_n)$ of $v_n\cdots v_2s$.
\end{enumerate}

We shall need the following later.
\begin{lemm}\label{Lemdcomcrisequ} Assume that $A_S$ is an Artin-Tits group.
If $(v_1,\cdots, v_n)$ is the sequence of factors of a rightward critical factorization, then so are
both $(v_1,\ldots,v_k)$ and $(sv_{k+1},\ldots,v_n)$, for any $ 1\leq k<n$,
where $s$ is the final letter of the last term of $\tau(v_1,\ldots,v_k)$.
\end{lemm}
\begin{proof} This follows from the definition. We leave the details to the reader.  
\end{proof}

\begin{defi} \label{defcrittrans} 
Where a freely reduced word $w$ admits a (rightward or leftward) critical factorization of length $n$, with critical factors $v_1,\ldots,v_n$,
the sequence of $n$ critical transformations that leads to the replacement of the (rightward or leftward) concatenation of the critical factors by the (rightward or leftward)
concatenation of the terms in the sequence $\tau(v_1,\ldots,v_n)$
is called a {\em critical sequence of transformations}. That sequence
is called rightward critical or leftward critical
depending on whether the factorization is rightward critical or leftward critical.
\end{defi}

We shall need the following result, derived from  \cite[Proposition 3.3]{HoR}
and \cite[Proposition 3.2 (2)]{HoR2}.  
\begin{prop}\label{propHoRcrseq}
Let $A_S$ be an Artin-Tits group of large type.
Suppose that $w$ is freely reduced over $S \cup S^{-1}$, representing the 
group element $g \in A_S$.
\begin{enumerate}
\item If $w$ is not geodesic then $w$ admits a rightward critical factorization
of the form $\alpha v_1\cdots v_n \beta$, with $\beta$ non-empty, and application of a single rightward critical sequence of transformations transforms $w$ to
a word $w'=\alpha v'_1\cdots v'_n\beta$ that is not freely reduced.
Specifically, the last letter of $v'_n$ is the inverse of the first letter
of $\beta$.
\item
Suppose that some lexicographic ordering has been chosen for $S \cup S^{-1}$,
and that $w_g$ is the minimal representative of $g$, with respect to the shortlex order on words

If $w$ is not equal to $w_g$, then some finite combination of rightward and leftward critical sequences of transformations together with free cancellation (after each rightward sequence) transforms $w$ to $w_g$.
\end{enumerate}
\end{prop}

We shall also need the following corollary later. 
\begin{coro} \label{coroequivred} Assume that $A_S$ is an Artin-Tits group of Large type, and that  $w,w'$ are two geodesic representative words of the same element. Then  $w' \goesR{1,2} w$. 
\end{coro}
\begin{proof}  Set $g = \overline{w} = \overline{w'}$,
fix an ordering of $S \cup S^{-1}$, and let $w_g$ be the unique minimal 
representative of $g$ with respect to the associated shortlex ordering on words.
Proposition~\ref{propHoRcrseq} ensures that a combination of leftward critical 
sequences of transformations
transforms each of $w,w'$ into $w_g$; hence a sequence of
transformations of type $\beta$ (according to Definition~\ref{D:Specialbeta}) 
transforms $w$ and $w'$ into $w_g$, and therefore  $w$ into $w'$. The result now follows by 
Lemma~\ref{Lemtauet3}. 
\end{proof}

\subsection{The $S$-transversal for large Type Artin-Tits groups} \label{secTransv}
With the notion of critical sequence at hand, we can now introduce an $S$-transversal for a large type Artin-Tits group and prove that Property $H^\#$ holds for Artin-Tits group of large type.  

To verify that the set we define does give an $S$-transversal, we shall need the following.
\begin{prop}\label{Proptrans} Assume that $A_S$ is an Artin-Tits group of large type. Suppose that $S_0 \subseteq S$ and $g \in A_S$. Then 
\begin{enumerate}
\item Choose $h \in A_{S_0}$ to satisfy the equation $g = g_0h$,
where $g_0$ is a minimal length element of $gA_{S_0}$.  Let $w_0$ and $z$
be geodesic representative words  of $g_0$ and $h$, respectively. Then $w_0z$ is a  geodesic representative word  of $g$. 
\item The left coset $gA_{S_0}$ contains a unique element of minimal length. \end{enumerate}
\end{prop}
\begin{proof} 
Let $g_0,h,w_0,z$ be as specified in (1).
It follows from the convexity (see Proposition~\ref{prop:ChP}) of parabolic subgroups of Artin groups that $z$ is a word written over $S_0 \cup S_0^{-1}$.
Since $w_0$ and $z$ are freely reduced words, $w_0z$ must also be freely reduced;
for otherwise $w_0$ has a proper suffix over $S_0 \cup S_0^{-1}$ that freely cancels
with a proper prefix of $z$, and hence a proper prefix of $w_0$ represents a
shorter element of $gA_{S_0}$ than $g_0$. 

So now suppose that $w_0z$ is not geodesic.
Then by Proposition~\ref{propHoRcrseq} $w_0z$  admits a 
rightward critical factorization $\alpha v_1\cdots v_n \beta = \alpha v \beta$
such that the last letter of the concatenation of the terms of $\tau(v_1,\ldots,v_n)$
freely cancels with the first letter $s$ of the non-empty word $\beta$.
Since $w_0$ and $z$ are geodesic, clearly $vs$ cannot be a factor of either.
If $v$ were a suffix of $w_0$, then we could replace that suffix by
the concatenation of the terms of $\tau(v_1,\ldots,v_n)$ to get a new choice of $w_0$; but then we would have free reduction
between $w_0$ and $z$.
We deduce that $v$ must overlap $w_0$ and $z$.

Now, if for some $k<n$ the first $k$ terms of the critical factorization
were within $w_0$, then by Lemma~\ref{Lemdcomcrisequ} those terms would give a critical
factorization of a factor of $w_0$, and replacing them
by the terms of $\tau(v_1,\ldots,v_k)$ would give us a different minimal
choice for $w_0$. With this choice,  Lemma~\ref{Lemdcomcrisequ} now
implies that $w_0z$ would admit a 
rightward critical factorization with critical factors $tv_{k+1},\ldots,v_n$,
where $t$  is the last letter of the concatenation of the terms of $\tau(v_1,\cdots,v_k)$.
Hence we can assume that the first term $v_1$ of the sequence overlaps
both $w_0$ and $z$.

Suppose now that $v_1$ involves generators $s_1$ and $t_1$. Let $s_1$ be the rightmost letter of $v_1$. Then we know that $s_1$ is within $z$, so $s_1 \in S_0 \cup S_0^{-1}$. But since $v_1$ is not a subword of $z$, we see that $t_1 \not \in S_0 \cup S_0^{-1}$.

Let $t$ be the generator immediately to the right of $v_1$ in $z$ (we know there
must be such a generator even when $n=1$, since $\beta$ is non-empty).
Then, by Lemma~\ref{Lemtauetgen},
$t_1$ is the name of the generator at the end of $\tau(v_1)$; so either
$t_1=t^{\pm 1}$ (when we have a free reduction), or $t_1$ is the name of a generator in $v_2$, which is a subword of $z$. Either way $t_1 \in S_0$, and so we have a contradiction.

So all cases lead to a contradiction. We conclude that $w_0z$ must be a geodesic representative word of $g$. 

So (1) holds. Now, (2) follows immediately. For if $g_1$ is another element of $gA_{S_0}$, 
then $g_1=g_0h_1$ for some non-trivial $h_1 \in A_{s_0}$, and so $g_0$ has  a geodesic
representative of which $w_0$ is a proper prefix.
\end{proof}

\begin{coro} \label{corotrans} Assume that $A_S$ is an Artin-Tits group of
Large type. For $S_0\subseteq S$, Denote by  $T(S_0)$ the set of elements of
$A_S$ that are of minimal length in their left $A_{S_0}$-coset.
Then, the sequence $(T(S_0))_{S_0\subseteq S}$, is a $S$-sequence of transversals. 
\end{coro}
\begin{proof} That each $T(S_0)$ is a left transversal follows immediately from part (2) of Proposition~\ref{Proptrans}.
And $1 \in T(S_0)$ is immediate since the  element $1$ is minimal in its left coset $A_{S_0}$. 
\end{proof}

\begin{prop}  \label{protrans} Let $A_S$ be an Artin-Tits group of Large type.
Then $A_S$ satisfies the Property $H^\#$.
\end{prop}
\begin{proof} 
Let $(T(S_0))_{S_0\subseteq S}$ be 
 the $S$-sequence of transversals defined in Corollary~\ref{corotrans}. 
Choose subsets $S_0,S'$ of $S$,
and let $w$ be a word with $\overline{w} \in A_{S'}$.
Choose $g_0$ in  $T(S_0)$  and $h$ in $A_{S_0}$ with $\overline{w} = g_0h$,
and let $w_0$ and $z$ be geodesic representative words of $g_0$ and $h$, respectively.  

Since  $h \in A_{S_0}$, the convexity of parabolic subgroups ensures that $z$ is written over $S_0 \cup S_0^{-1}$.  
By Proposition~\ref{Proptrans}, the word $w_0z$ is a geodesic representative word of  $\overline{w}$, and so is written over
$S' \cup {S'}^{-1}$. 
Hence, $w_0$ is written over  $S' \cup {S'}^{-1}$, and $z$ is written over  $S_0' \cup {S_0'}^{-1}$,
where $S_0'  = S'\cap S_0$.
So, in order to prove that Property $H^\#$ holds for $A_S$, it remains to prove that $w \gozotR  w_0z$.

By proposition ~\ref{PropHLarge2}, there exists a geodesic representative word $w'$ so that  $w \gozotR w'$.
It remains to show that $w' \goesR{1,2} w_0z$.  Since $w'$ and  $w_0z$ are two
geodesic representative words of the same element, this follows immediately
from Corollary~\ref{coroequivred}. \end{proof}

\subsection{Proof of Theorems~\ref{TH2} and ~\ref{TH1}}
\label{sec:finalproof}
We are now almost ready to prove Theorem~\ref{TH2}.  Using
Proposition~\ref{propHDpAmal},   we deduce the result from Proposition~\ref{protrans}  by induction on the
number of edges labelled with $\infty$ in the Coxeter graph. We need a
preliminary result
\begin{lemm} \label{Lemmdirectprod}Let $(S_1,R_1)$ and $(S_2,R_2)$ be two Artin-Tits presentations that satisfy property $H^\#$ with $S_1\cap S_2 = \emptyset$.  Set $$R_{1,2} = \{s_1s_2 = s_2s_1\mid s_1\in S_1;\ s_2\in S_2\}$$ Then, the Artin-Tits presentation $(S_1\cup S_2,R_1\cup R_2\cup R_{1,2})$  satisfies property $H^\#$.                                                                                              
\end{lemm} 
Note that in the above lemma, the Artin-Tits group $A_{S_1\cup S_2}$ associated with the presentation $(S_1\cup S_2,R_1\cup R_2\cup R_{1,2})$ is isomorphic to the direct product $A_{S_1}\times A_{S_2}$.
\begin{proof} Let  $(T(S_{1,0}))_{S_{1,0}\subseteq S_1}$, be a $S_1$-sequence
of transversals and  $(U(S_{2,0}))_{S_{2,0}\subseteq S_2}$, be a $S_2$-sequence of transversals.
For $S_0\subseteq S_1\cup S_2$, set  $$V(S_0) =\{g_1g_2\mid g_1\in T(S_0\cap S_1);  g_2\in U(S_0\cap S_2)\}.$$
Since  $A_{S_1\cup S_2}$ is the direct product of  $A_{S_1}$ and $A_{S_2}$, it immediately follows that $(V(S_{0}))_{S_{0}\subseteq S}$, is a $S$-sequence of transversals.

Let $S'$ be a subset of $S$.
Set $S'_1 = S_1\cap S'$, $S'_2 = S_2\cap S'$, $S'_{1,0} = S_1\cap S'\cap S_0$ and $S'_{2,0} = S_2\cap S'\cap S_0$. 
Let $w$ be a word in $S\cup S^{-1}$ so that $\overline{w}$ belongs to $A_{S'}$. 

Since the latter group is the direct product of  $A_{S'_1}$ and $A_{S'_2}$, 
there exists a unique pair $(g_1g_2)$ so that $\overline{w} = g_1g_2$ with
$g_1\in A_{S'_1}$ and $g_2\in A_{S'_2}$.   
Note that for any two words $w',w''$, any $s_1$ in $S_1$ and $s_2$ in $S_2$,
and $\varepsilon,\tau\in \{\pm 1\}$, any transformation from  $w's^\varepsilon_is_j^\tau w"$
to $w's^\tau_js_i^\varepsilon w"$ is an elementary transformation of type~$1$ or of type~$2$,
since $s_1s_2 = s_2s_1$ belong to $R_{1,2}$  Therefore, there exist $w_1$ in $S_1\cup S_1^{-1}$ and $w_2$ in $S_2\cup S_2^{-1}$ so that $w\goesR{1,2} w_1w_2$. 
This imposes  $\overline{w_1} = g_1$ and  $\overline{w_2} = g_2$.  

Since $(S_1,R_1)$  satisfies property $H^\#$, there exist $v_1 \in (S_1'\cup {S_1'}^{-1})^*$ and $u_1 \in (S'_{1,0}\cup {S'_{1,0}}^{-1})^*$  so that $w_1 \gozotR v_1 u_1$ and $\overline{v_1} \in T(S_{1,0})$. 
Similarly, there exist $v_2 \in (S_2'\cup {S_2'}^{-1})^*$ and $u_2 \in (S'_{2,0}\cup {S'_{2,0}}^{-1})^*$  so that $w_2 \gozotR v_2 u_2$ and $\overline{v_2} \in T(S_{2,0})$. 
Thus, we get  $w\gozotR v_1u_1v_2u_2$.  
By the same argument 	as before we have  $u_1v_2 \goesR{1,2} v_2u_1$  and   $w\gozotR v_1v_2u_1u_2$. 
The word $v_1v_2$ lies in $(S'\cup {S'}^{-1})^*$ and by definition of $V(S_{0})$, the element $\overline{v_1v_2}$ belongs to this transversal. 
Finally, the word $u_1u_2$ lies in  $(S'_{0}\cup {S'_{0}}^{-1})^*$. 
Hence, Property $H^\#$ holds for the presentation $(S_1\cup S_2,R_1\cup R_2\cup R_{1,2})$. 
\end{proof}

\begin{proof}[Proof of Theorem~\ref{TH2}] 
Recall that $\mathcal{T}$ is the family of those Artin-Tits presentation $(S,R)$ whose Coxeter graph~$\Gamma$ satisfies the property that every connected full subgraph without an $\infty$-labelled edge is either of spherical type or of large type. 
For $(S,R)$ in~$\mathcal{T}$, we denote the number of edges labelled with $\infty$ in the associated graph by  $d_\infty(S,R)$. 
Let $(S,R)$ belong to $\mathcal{T}$ and denote by $\Gamma$ its Coxeter graph. 
 As announced in the introduction of this section, we prove the result by induction on $d_\infty(S,R)$.   

If $d_\infty(S,R) = 0$, then $T$ is a union of connected Coxeter graphs of  spherical type or of large type. 
Since Property $H^\#$ holds for both types (see \cite{DeG} for the spherical case), by Lemma~\ref{Lemmdirectprod},  $(S,R)$ satisfies   Property $H^\#$. 

Now assume $d_\infty(S,R) \geq 1$.  
Let $s_1,s_2$ belong to $S$ so that $m_{s_1,s_2}=\infty$. 
Set $S_1 = S\setminus \{s_2\}$,  $S_2 = S\setminus \{s_1\}$ and $S_{1,2} = S\setminus \{s_1,s_2\}$.  
Consider $(S_1, R_1)$, $(S_2, R_2)$ and $(S_{1,2}, R_{1,2})$, the Artin-Tits presentations associated with  the full subgraphs of $\Gamma$ generated by $S_1$, $S_2$ and $S_{1,2}$, respectively. 
They are the presentations of the subgroups $A_{1}$, $A_{2}$ and $A_{1,2}$ of $A_S$ generated by $S_1$, $S_2$ and $S_{1,2}$, respectively, and $A_S$ is equal to the amalgamated product $A_1*_{1,2}A_2$. 
Since $(S,R)$ belongs to $\mathcal{T}$, the presentations $(S_1, R_1)$, $(S_2, R_2)$ and $(S_{1,2}, R_{1,2})$ belong to $\mathcal{T}$, too. 
Moreover, we have $d_\infty(S_1,R_1) <d_\infty(S,R)$, $d_\infty(S_2,R_2) <d_\infty(S,R)$ and  $d_\infty(S_{1,2},R_{1,2}) <d_\infty(S,R)$. 
So by the induction hypothesis,  Property $H^\#$ hold for  $(S_1, R_1)$, $(S_2, R_2)$ and $(S_{1,2}, R_{1,2})$.  
So, $(S,R)$ verifies Property $H^\#$  by Proposition~\ref{propHDpAmal}. It also satisfies Property $H$ by~\ref{propHDpAmal}.
\end{proof}

Theorem~\ref{TH1} can be deduced as an immediate corollary, using Proposition~\ref{propcarSLAT} and the fact that
every Artin-Tits presentation of large type belongs to the family~$\mathcal{T}$.

\end{document}